\documentclass{birkjour_t2}

\usepackage{amsmath, amsthm, amssymb, dsfont, epsfig, graphicx, wrapfig,xcolor}

\newtheorem{thm}{Theorem}[section]

\newtheorem{lem}[thm]{Lemma}

\numberwithin{equation}{section}
\newtheorem{exam}{Example}[section]

\makeatletter
\newcommand{\rom}[1]{\expandafter\@slowromancap\romannumeral #1@}
\makeatother

\begin{document}

\title{Analysis of PML Method for Stochastic Convected Helmholtz Equation}

\author{Sang-Hyeon Park}


\address{National Institute for Mathematical Sciences, Jeonmin-dong 463-1, Yuseong-gu, 305-811 Daejeon , Republic of Korea}

\author{Imbo Sim}


\address{National Institute for Mathematical Sciences, Jeonmin-dong 463-1, Yuseong-gu, 305-811 Daejeon , Republic of Korea}
\email{imbosim@nims.re.kr}
\thanks{Corresponding author : Imbo Sim}


\date{January 1, 2004}

\begin{abstract}
We propose and analyze the perfectly matched layer (PML) method for the time-harmonic acoustic waves driven by the white noise source in the presence of the uniform flow. A PML is an artificial absorbing layer commonly used to truncate computational regions to solve problems in unbounded domains. We study a modification of PML method based on B\'ecache et. al. \cite{BBL2004}. A truncated domain problem for stochastic convected Helmholtz equation in the infinite duct is constructed by applying PMLs. Our PML method omits the instability of inverse upstream modes in the PML. Moreover, a suitable jump condition on boundaries between computational domain and PMLs is not required. We analyze the stochastic error generated by truncations of the domain. Thus the convergence analysis of the solution is provided in the sense of mean-square.

\end{abstract}



\maketitle

\section{Introduction}

We study the stochastic convected Helmholtz equation with the white noise source in the infinite duct. Let the duct be $\Omega=\{x=(x_1,x_2)|\;x_1\in\mathds{R}, x_2\in(0,d)\}\subset\mathds{R}^2$, where $d$ is a positive constant. The acoustic pressure field $p(x)$ in the presence of a uniform flow satisfies the convected Helmholtz equation in $\Omega$:
\begin{equation}\label{oe1}
(1-M^2)\frac{\partial^2 p}{\partial x^2_1}+\frac{\partial^2 p}{\partial x^2_2}+2ikM\frac{\partial p}{\partial x_1}+k^2p=f+1_{\Omega_f}\dot{W}\;\;\;\;\;\; \text{in } \Omega,
\end{equation}
where the Mach number $M=\frac{v_0}{c_0}$, the wave number $k=\frac{\omega}{c_0}$, and $c_0$ sound velocity in the fluid. In addition, the pressure satisfies the homogeneous Neumann boundary condition on the wall,
\begin{equation}\label{oe2}
\frac{\partial p}{\partial x_2}=0\;\;\;\text{on }\partial\Omega.
\end{equation}
We assume that $0\leq M<1$. The right-hand side of (\ref{oe1}) consists of deterministic term $f$ and stochastic term $\dot{W}(x)$. The stochastic term $\dot{W}(x)$ is the formal derivative of the white noise in space \cite{W1986}. We suppose that the right-hand side of (\ref{oe1}) has a compact support $\Omega_f$ in $\Omega$ i.e. $\Omega_f=\mathrm{supp}(f)$ and $\Omega_f\subset \Omega$. Here, $f\in L^2(\Omega)$ is assumed. The function $1_{\Omega_f}(x)$ is an indicator function.

Convected Helmholtz equation describes the time-harmonic acoustic waves generated by moving media. In the presence of a uniform flow, time-harmonic linearized Euler equations reduce to the convected Helmholtz equation for the pressure. The presence of a mean flow causes the mathematical treatment of the problem much more difficult, since phase and group velocities of the acoustic wave could have opposite signs \cite{BGH2010, DJ2006, HHT2003}.

The main goal of this paper is to construct a modified perfectly matched layer(PML) method for the problem (\ref{oe1}) and to analyze a suitable error estimation in the sense of $E[\|\cdot\|^2_{L^2}]$.

Perfectly matched layer was introduced by B\'erenger in \cite{B1994} for the time-dependent Maxwell equations. To design an efficient absorbing boundary condition for the time harmonic acoustic waves in unbounded domains, the PML method is usually used. In the presence of a flow, compared to the classical waves, PMLs could cause instabilities of the solution. The presence of instabilities have been studied in \cite{BFJ2003} by using group and phase velocities. In \cite{BBL2004}, B\'ecache et. al. studied a PML formulation for the convected Helmholtz equation in a duct to avoid instabilities of the solution. They used a substitution $\frac{\partial}{\partial x_1}\rightarrow \alpha\frac{\partial}{\partial x_1}+i\lambda$. In \cite{BBL2004}, $\lambda$ is chosen as a constant $\frac{-kM}{1-M^2}$ in the PMLs. The complex-valued function $\alpha$ satisfies $\mathrm{Re}(\alpha)>0$ and $\mathrm{Im}(\alpha)<0$ in PMLs. This approach provides a successful analysis of error estimation under a suitable jump condition between the computational domain and PMLs.

According to \cite{ANZ1998,DG2001,DT2002}, regularity estimates for SPDEs are usually weak which leads to low order error estimates. To overcome this complication, authors in \cite{ANZ1998, CYY2007, CZZ2008,DG2001, DT2002} considered the discretized white noise instead of $\dot{W}$, since it is more regular. Especially, Cao et. al. \cite{CZZ2008} established an error estimation of the PML method for Helmholtz equation by applying the discretized white noise. We also utilize the discretized white noise to obtain the desired regularity for the error estimation as in \cite{CZZ2008}. Replacing the white noise of the forcing term with a discretized version, an approximated problem is built as follows.
\begin{equation}\begin{split}\label{123}
&(1-M^2)\frac{\partial^2 p^{h}}{\partial x^2_1}+\frac{\partial^2 p^{h}}{\partial x^2_2}+2ikM\frac{\partial p^{h}}{\partial x_1}+k^2p^{h}=f(x)+\dot{W}^{h}(x)\;\;\;\;\;\; \text{in }\Omega,\\
&\frac{\partial p^{h}}{\partial x_2}=0\;\;\;\;\;\; \text{on }\partial\Omega,
\end{split}\end{equation}
where $\dot{W}^{h}(x)$ is the discretized white noise. Since the problem (\ref{123}) is more regular than the problem (\ref{oe1}), the convergence analysis can be carried out for the PML method as in \cite{CZZ2008}.

In this setting (\ref{123}), we study a modified PML method instead of that in \cite{BBL2004}. The modified PML formulation considered in this paper is described in the following. The left-hand side of the problem (\ref{123}) is reformulated by
\begin{equation*}\begin{split}
&(1-M^2)(\frac{\partial }{\partial x_1}+\frac{Mki}{1-M^2})^2p^{h}+\frac{\partial^2 p^{h}}{\partial x^2_2}+\frac{k^2}{1-M^2}p^{h}.
\end{split}\end{equation*}
Then a substitution
\begin{equation*}
\frac{\partial }{\partial x_1}+\frac{Mki}{1-M^2} \rightarrow \alpha(x_1)(\frac{\partial }{\partial x_1}+\frac{Mki}{1-M^2})
\end{equation*}
is applied to construct a PML problem \cite{S2012}. Here, the complex function $\alpha(x_1)$ is given by
\begin{equation*}\begin{split}
\alpha(x_1)=\frac{-i\omega}{-i\omega+\sigma(x_1)},
\end{split}\end{equation*}
where $\omega>0$ and $\sigma(x_1)$ is a positive real-valued function in $\mathcal{C}^{1}(\mathds{R})$ such that $\sigma(x_1)=0$ in the computational domain. Therefore, $p^{h}$ is continuously connected on the boundary between the computational domain and PMLs. In this reason, this PML does not require a jump condition. It also follows from the modal expansion approach in \cite{BBL2004} that a PML-truncation error of this model.

The outline of this paper is as follows. In Section 2, we discuss the discretized white noise. The error estimation $E[\|p-p^{h}\|^{2}_{L^2}]$ and convergence analysis are provided.  In Section 3, we discuss about a modified PML method and the convergence of the problem as the length of layer goes to infinity. In Section 3.1, problems are restricted in a bounded domain $\Omega_b$ in order to analyze an error generated by PML . In section 3.2, we derive $\|\cdot\|_{H^{1}}$-error of the PML model. Combining two error estimations in section 2 and section 3.2, we finally establish the total error analysis in the sense of $E[\|\cdot\|^2_{L^2}]$.

\section{Approximation driven by the Discretized White Noise}

In this section, we derive the approximated problem (\ref{123}) from (\ref{oe1}). This approximation provides more regular solution than the solution of (\ref{oe1}). Let $\mathcal{T}$ be a trianglation $\bigcup^{N}_{i=1}K_i\subset \Omega_f$. Each element $K_i$ has at most one curved edge aligned with $\Omega$. Let $h_i=\mathrm{diam}\; K_i$, $h=\mathrm{max}_{i}\;h_i$ and $r_i$ be the radius of the largest ball inscribed $K_i$. We assume that $\mathcal{T}$ is quasi-uniform. The random variable $\xi_i$ is define by
\begin{equation}
\xi_i:=\frac{1}{\sqrt{|K_i|}}\int_{K_i}dW(x),\;\; 1\leq i\leq N.
\end{equation}
Here, $|K_i|$ denotes the area of $K_i$. According to \cite{W1986}, the family of random variables $\{\xi_i\}$ is independent identically distributed (i.i.d.) normal random variables with mean 0 and variance 1. The piecewise constant approximation $\dot{W}^{h}(x)\in L^{2}(\Omega)$ is defined by
\begin{equation}
\dot{W}^{h}(x)=\sum^{N}_{i=1}\frac{1}{\sqrt{|K_i|}}\xi_i 1_{K_i}(x)
\end{equation}
which has following properties:
\begin{equation*}\begin{split}
E[\dot{W}^{h}(x)\dot{W}^{h}(y)]&=\frac{1}{|K_i|},\;\; \text{ for $x,y\in K_i$},\\
E[\dot{W}^{h}(x)\dot{W}^{h}(y)]&=0,\;\;\;\;\;\;\;\;\; \text{for $x\in K_i$ and $y\in K_j$, $i\neq j$}.
\end{split}\end{equation*}
Therefore, the problem (\ref{oe1}) can be reformulated by replacing $\dot{W}$ with $\dot{W}^{h}$. Consider the following approximated problem
\begin{equation}\begin{split}\label{oed1}
&(1-M^2)\frac{\partial^2 p^{h}}{\partial x^2_1}+\frac{\partial^2 p^{h}}{\partial x^2_2}+2ikM\frac{\partial p^{h}}{\partial x_1}+k^2p^{h}=f+\dot{W}^{h}\;\;\;\;\;\; \text{in }\Omega,\\
&\frac{\partial p^{h}}{\partial x_2}=0\;\;\;\;\;\; \text{on }\partial\Omega.
\end{split}\end{equation}
The problem (\ref{oed1}) is more regular than the problem (\ref{oe1}), since $\dot{W}^{h}(x)\in L^{2}(\Omega)$. Moreover, a weak form of the problem (\ref{oed1}) is
\begin{equation}\begin{split}\label{woed1}
a_{\Omega}(p^{h},q)=-\int_{\Omega}(f+\dot{W}^{h})\bar{q}\;dx,\;\;\;\;\; \forall q\in H^{1}(\Omega),
\end{split}\end{equation}
where the sesquilinear form $a_{\Omega}(\cdot,\cdot)$ is
\begin{equation}\begin{split}\label{uw1}
a_{\Omega}(p,q)=\int_{\Omega}(\mathcal{M}\;\nabla p)\cdot\nabla\bar{q}\;dx+\int_{\Omega}k^2 p\bar{q}\;dx-2ik\int_{\Omega}\nabla\cdot \left(\begin{array}{c} Mp\\0 \\\end{array}\right)\bar{q}\;dx.
\end{split}\end{equation}
where $\mathcal{M}=$
$\left(\begin{array}{cc}
1-M^2 & 0\\ 0 & 1\\
\end{array}\right)$.

Now, we prove that the solution $p^h$ of (\ref{oed1}) converges to $p$ in a given bounded domain $\Omega_b\in\Omega$. Using the stochastic integration formula in \cite{BS1994} and \cite{W1986}, the mean-square error $E[\|p-p^{h}\|^2_{L^2(\Omega_b)}]$ is derived in the following. Let $G(x,y)$ be the Green function of the convected Helmholtz equation (\ref{oed1}). We denote $H^{(1)}_0(x)$ is a Hankel function of the first kind of order $0$. By the method of images(\cite{K1999}), for a fixed $y\in\Omega$, the Green function $G(x,y)$ is written by
\begin{equation}\begin{split}\label{gf}
G(x,y)=&\Phi(x,y)+\Phi(x,-y)\\
&+\sum^{\infty}_{n=1}\Big(\Phi(x,e^{+}_{n}-y)+\Phi(x,e^{+}_{n}+y)+\Phi(x,e^{-}_{n}-y)+\Phi(x,e^{-}_{n}+y)\Big),
\end{split}\end{equation}
where
\begin{equation}\begin{split}\label{g1}
&\Phi(x,y)=\frac{i}{4\sqrt{1-M^2}}H^{(1)}_0(k\rho(x-y))e^{i\frac{M}{1-M^2}(x_1-y_1)},\\
&\rho(x)=\frac{\sqrt{x^2_1+(1-M^2)x^2_2}}{1-M^2},\\
&e^{\pm}_{n}=(0,\pm 2dn).
\end{split}\end{equation}
Accrording to \cite{BS1994} and \cite{W1986}, $p$ and $p^{h}$ can be written by
\begin{equation}\begin{split}\label{gs1}
p(x)=&\int_{\Omega}G(x,y)f(y)\;dy+\int_{\Omega}G(x,y)1_{\Omega_f}(y)\;dW(y),\\
p^{h}(x)=&\int_{\Omega}G(x,y)f(y)\;dy+\int_{\Omega}G(x,y)1_{\Omega_f}(y)\;dW^{h}(y).
\end{split}\end{equation}
To analyze the error with respect to $p$ and $p^{h}$, we start with following two lemmas.

\begin{lem}\label{lem1}
There exists a Lipschitz continuous function $\tilde{V}(x,y)$ with respect to $x$ and $y$ such that $\Phi(x,y)$ is decomposed by
\begin{equation}\begin{split}
\Phi(x,y)=\frac{1}{2\pi\sqrt{1-M^2}}\ln\frac{1}{k\rho(x-y)}e^{-i\frac{M}{1-M^2}k(x_1-y_1)}+\tilde{V}(x,y).
\end{split}\end{equation}
\end{lem}
\begin{proof}
Let us define a linear map $T:\mathds{R}^2\rightarrow\mathds{R}^2$ by $T(x)=(\frac{x_1}{1-M^2},\frac{x_2}{\sqrt{1-M^2}})$. Let $\tilde{x}=T(x)$ and $\tilde{y}=T(y)$. Then it follows from (\ref{g1}) that
\begin{equation}\begin{split}
\rho(x-y)=|\tilde{x}-\tilde{y}|.
\end{split}\end{equation}
It is well-known that there exists a Lipschitz continuous function $V(x,y)$ with respect to $x$ and $y$ such that
\begin{equation}\begin{split}\label{decom1}
iH^{1}_0(k|\tilde{x}-\tilde{y}|)=\frac{1}{2\pi}\ln\frac{1}{|\tilde{x}-\tilde{y}|}+V(\tilde{x},\tilde{y})
\end{split}\end{equation}
(refer to \cite{CZZ2008}). Since the map $T$ is linear, $\tilde{V}(x,y):=V(T(x),T(x))e^{-i\frac{M}{1-M^2}k(x_1-y_1)}$ is still Lipschitz continuous with respect to $x$ and $y$. Consequently, the proof is completed from (\ref{g1}) and (\ref{decom1}).
\end{proof}

From (\ref{gf}) and Lemma \ref{lem1}, the next Lemma is proved.
\begin{lem}\label{lem2}
Let $\Omega_b$ be an open bounded domain in $\mathds{R}^2$. We assume that  $\Omega_f\subset\Omega_b\subset\Omega$. Suppose $y,z\in K_i$ and $h$ is sufficiently small. Then we have
\begin{equation}\begin{split}\label{ien1230}
\int_{\Omega_b}|G(x,y)-G(x,z)|^2 dx\leq C |y-z|^{2-\epsilon},
\end{split}\end{equation}
where $0<\epsilon<1$ and $C$ is a positive constant independent of $\epsilon$.
\end{lem}
\begin{proof}
Refer to the proof in the appendix.
\end{proof}

Applying Lemma \ref{lem2}, we are now in position to obtain the error estimation $E[\|p-p^{h}\|^2_{L^2(\Omega_b)}]$.
\begin{thm}\label{thm1}
 Let $p$ and $p^{h}$ are the solution of (\ref{oe1}), (\ref{oe2}), and (\ref{oed1}), respectively. Let $\Omega_b$ be an open bounded domain in $\mathds{R}^2$. We assume that  $\Omega_f\subset\Omega_b\subset\Omega$. Suppose $y,z\in K_i$ and $h$ is sufficiently small. Then we obtain
\begin{equation}\begin{split}\label{finalerror1}
E[\|p-&p^{h}\|^2_{L^2(\Omega_b)}]\leq C h^{2-\epsilon},
\end{split}\end{equation}
where $0<\epsilon<1$ and $C$ is a positive constant independent of $\epsilon$ and $h$.
\end{thm}
\begin{proof}
By (\ref{gs1}) and It\^o isometry, it follows that
\begin{equation}\begin{split}
E[\|p-&p^{h}\|^2_{L^2(\Omega_b)}]=E[\|\int_{\Omega_f}G(x,y)dW(y)-G(x,y)dW^{h}(y)\|^{2}_{L^2(\Omega_b)}]\\
&=E[\int_{\Omega_b}|\int_{\Omega_f}G(x,y)dW(y)-G(x,y)dW^{h}(y)|^{2}dx]\\
&=E[\int_{\Omega_b}|\sum^{N}_{i=1}\int_{K_i}G(x,y)dW(y)-\|K_i\|^{-1}\sum^{n_1}_{i=1}\int_{K_i}G(x,z)dz\int_{K_i}dW^{}(y))|^{2}dx]\\
&=E[\int_{\Omega_b}|\sum^{N}_{i=1}\sum^{N}_{j=1}\int_{K_i}\int_{K_j}\|K_i\|^{-1}(G(x,y)-G(x,z))dz\;dW(y)|^{2}dx]\\
&\leq C\sum^{N}_{i=1}\sum^{N}_{j=1}\int_{K_i}\int_{K_j}\int_{\Omega_b}|G(x,y)-G(x,z)|^2\;dxdzdy,
\end{split}\end{equation}
where positive constant $C$ depending on $\sup_{i}|K_i|$. Thus, Lemma \ref{lem2} leads to the completeness of the proof
\end{proof}

Theorem \ref{thm1} implies that $p^{h}$ converges to $p$ as $h\rightarrow 0$ in a bounded domain $\Omega_b$. In Section 3, $\Omega_b$ will be choosen as a computational domain.

\section{PML model for convected Helmholtz equation}

Let us study the modified PML method which is described in Section 1. We first prove that the solution of PML model converges to the solution of (\ref{woed1}) in $\Omega_b$. Then, it follows from combining with the result in Theorem \ref{thm1} that the stochastic error estimation between solutions of (\ref{oe1}) and PML model.

As in \cite{BBL2004}, the presence of inverse upstream modes should be considered to avoid an unstable solution. The presence of instabilities was studied in \cite{BFJ2003} by using group velocities. We recall the modified PML model as follows. The problem (\ref{oed1}) is reformulated by
\begin{equation}\begin{split}\label{oedreduce2}
&(1-M^2)(\frac{\partial }{\partial x_1}+\frac{Mki}{1-M^2})^2p^{h}+\frac{\partial^2 p^{h}}{\partial x^2_2}+\frac{k^2}{1-M^2}p^{h}=f+\dot{W}^{h}\;\;\;\; \text{in }\Omega.
\end{split}\end{equation}
We now apply the following substitution \cite{S2012}:
\begin{equation}\label{PMLtrans}
\frac{\partial }{\partial x_1}+\frac{Mki}{1-M^2} \rightarrow \alpha(x_1)(\frac{\partial }{\partial x_1}+\frac{Mki}{1-M^2}).
\end{equation}
Here, the complex function $\alpha(x_1)$ is given by
\begin{equation}\begin{split}\label{alpha}
\alpha(x_1)=\frac{-i\omega}{-i\omega+\sigma(x_1)},
\end{split}\end{equation}
where $\omega>0$ and $\sigma(x_1)$ is a positive real-valued function in $\mathcal{C}^{1}(\mathds{R})$.
\begin{exam}\label{exam1}
Let
\begin{equation}
\sigma(x_1)=\Big\{
\begin{array}{l}
0\;\;\;\;\;\;\;\;\;\;\;\;\;\;\;\;\;\;\;\;\;\;\;\;\;\;\;\;\;\text{if $x_1\in(x^{-},x^{+})$,}\\
\sigma_{+}\cdot (x_1-x^{+})^{2}\;\;\;\;\;\;\;\text{if $x_1>x^{+}$,}\\
\sigma_{-}\cdot (x^{-}-x_1)^{2}\;\;\;\;\;\;\;\text{if $x_1<x^{-}$,}\\
\end{array}
\end{equation}
where $\sigma_{\pm}>0$. This type of $\sigma$ is usually applied for acoustic wave problems. Here, constants $\sigma_{\pm}$ are associated with the reflection coefficient. Refer to \cite{C2002, J2012} for details.
\end{exam}
At this point, we have considered infinite PMLs, $\{x\in\Omega|\;x_1>x^{+}\}$ and $\{x\in\Omega|\;x_1<x^{-}\}$. In practice, we need to build bounded layers of the finite length $L$. Thus, let us define two PMLs, $\Omega^{L}_{+}=\{x\in\Omega^{L}|\;x^{+}<x_1<x^{+} + L\}$ and $\Omega^{L}_{-}=\{x\in\Omega^{L}|\;x^{-} - L<x_1<x^{-} \}$. Recall that we defined the computational domain $\Omega_b$ in Section 3. Denote that $\Omega^{L}:=\{x\in\Omega|\; x^{-}-L<x_1<x^{+}+L\}$ is the truncated domain and $\Sigma^{L}_{\pm}:=\{x\in\Omega|\;x_1=x^{\pm}\pm L\}$ is the external boundaries. For simplicity, we consider the Dirichlet boundary conditions on $\Sigma^{L}_{\pm}$. Let $p^{L}$ be the solution which satisfies
\begin{equation}\begin{split}\label{oedreduce3}
&(1-M^2)\left(\alpha(x_1)\frac{\partial }{\partial x_1}+\alpha(x_1)\frac{Mki}{1-M^2}\right)^2p^{L}+\frac{\partial^2 p^{L}}{\partial x^2_2}+\frac{k^2}{1-M^2}p^{L}=f+\dot{W}^{h}\;\;\;\;\;\; \text{in }\Omega_{L},\\
&\frac{\partial p^{L}}{\partial x_2}=0\;\;\;\;\;\; \text{on }\partial\Omega \cap \partial\Omega_{L},\\
&p^{L}=0\;\;\;\; \text{on }\Sigma^{L}_{\pm}.
\end{split}\end{equation}
Let $V_L=\{q\in H^{1}(\Omega_L)|\;q=0\; \mathrm{on}\; \Sigma^{L}_{\pm}\}$. A variational formula of (\ref{oedreduce3}) is written by
\begin{equation}\begin{split}\label{woed2}
a_{\Omega^{L}}(p,q)=-\int_{\Omega^{L}}(f+\dot{W}^{h})\bar{q}\;dx,\;\;\;\;\; \forall q\in V_{L},
\end{split}\end{equation}
where the form $a_{\Omega^{L}}(\cdot,\cdot)$ is
\begin{equation}\begin{split}
a_{\Omega^{L}}(p,q)=b_{\Omega^{L}}(p,q)+c_{\Omega^{L}}(p,q)
\end{split}\end{equation}
with
\begin{equation}\begin{split}
b_{\Omega_{L}}(p,q)&=\int_{\Omega_L}(\mathcal{M}_{\alpha}\;\nabla p)\cdot\nabla\bar{q}dx-\int_{\Omega_L}\frac{1}{\alpha(x_1)}\frac{k^2(1-\alpha^2 M^2)}{1-M^2}p\bar{q}\;dx,\\
c_{\Omega_L}(p,q)&=-2ik\int_{\Omega_L}\alpha(x_1)\nabla\cdot \left(\begin{array}{c} Mp\\0 \\\end{array}\right)\bar{q}\;dx-ikM\int_{\Omega_L}\alpha^{\prime}(x_1)p\bar{q}\;dx,
\end{split}\end{equation}
where $\mathcal{M}_{\alpha}=$
$\left(\begin{array}{cc}
(1-M^2)\alpha & 0\\
0 & \frac{1}{\alpha}\\
\end{array}\right)$.

Now we prove the following theorem.
\begin{thm}\label{thmwell} The problem (\ref{woed2}) is of Fredholm type.
\end{thm}
\begin{proof}
The idea of the proof is based on \cite{BBL2004} and \cite{BDLP2002}. First, we check that there exist a bounded operator $K_{L}$ on $H^1(\Omega^{L})$ such that
\begin{equation*}
(K_{L}p,q)_{H^1(\Omega^{L})}=c_{\Omega^{L}}(p,q)-\int_{\Omega_L}(\frac{1}{\alpha(x_1)}\frac{k^2(1-\alpha^2 M^2)}{1-M^2}+1)p\bar{q}\;dx,\;\;\;\;\;\forall\;p,q\in H^1(\Omega^{L}).
\end{equation*}
By the compactness of the embedding of $H^1(\Omega^{L})$ into $L^2(\Omega^{L})$, the operator $C_{L}$ is compact. Let us define $s_{\Omega_{L}}(p,q)=a_{\Omega^{L}}(p,q)-(K_{L}p,q)_{H^1(\Omega^{L})}$. Taking the real part of $s_{\Omega_{L}}(q,q)$, it is derived that
\begin{equation}\begin{split}
\mathrm{Re}(s_{\Omega_{L}}(q,q))=&\int_{\Omega_L}\left((1-M^2)\mathrm{Re}(\alpha)\left|\frac{\partial q}{\partial x_1}\right|^2+\mathrm{Re}(\frac{1}{\alpha})\left|\frac{\partial q}{\partial x_2}\right|^2+|q|^2\right)dx\\
\geq&\inf_{x\in\Omega_L}\left(\frac{(1-M^2)\omega^2}{\sigma^2(x_1)+\omega^2}\right)\int_{\Omega_L}|\nabla q|^2\;dx + \|q\|^2_{L^2(\Omega_L)}\\
\geq&C\|q\|^2_{H^1(\Omega_L)},
\end{split}\end{equation}
where $C=\inf_{x\in\Omega_L}\left(\frac{(1-M^2)\omega^2}{\sigma^2(x_1)+\omega^2}\right)$. This implies that $|s_{\Omega_{L}}(q,q)|\geq C\|q\|^2_{H^1(\Omega_L)}$. Thus, applying Lax-Milgram theorem and Riez representation theorem, there exist a bounded operator $S_{L}$ on $H^1(\Omega^{L})$ such that
\begin{equation*}
(S_{L}p,q)_{H^1(\Omega^{L})}=s_{\Omega_{L}}(p,q)\;\;\;\;\;\forall\;p,q\in H^1(\Omega^{L}),
\end{equation*}
and $S_{L}$ is of Fredholm type. Therefore, from Corollary 4.47 in \cite{AA2002}, $S_{L}+C_{L}$ is of Fredholm type.
\end{proof}
Since (\ref{woed2}) is of Fredholm type, the problem (\ref{woed2}) is well posed if and only if the homogeneous problem has only the trivial solution $p=0$. In order to ensure the well-posedness, it will be discussed in Section 3.2.

Next, we discuss reduced problems which are equivalent to (\ref{oed1}) and (\ref{oedreduce3}). These reduced problems have two main advantages: first, they provide a result of existence and uniqueness of the solution. Second, numerical methods can be used to solve them, since they are posed in a bounded domain $\Omega_b$. In this setting, the error analysis of the PML method is established in a bounded domain $\Omega_b$.

\subsection{Reduction to a bounded domain}
Let us reduce the problem (\ref{oed1}) in $\Omega$ to a problem in a bounded domain for the error estimation. The idea is based on that in \cite{BBL2004}. To derive the appropriate boundary conditions for the reduced problem, we first consider the homogeneous problem of (\ref{oed1}), i.e. the right-hand side is $0$. In this case, the solution of (\ref{oed1}) is represented by
\begin{equation}\begin{split}\label{mode1}
p^{\pm}_n(x)=e^{i\beta^{\pm}_n x_1} \phi_{n}(x_2),
\end{split}\end{equation}
where $n\in\{0\}\cup\mathds{N}$, $\phi_{0}(x_2)=\frac{1}{\sqrt{d}}$, $\phi_{n}(x)=\sqrt{\frac{2}{d}}\cos(\frac{n\pi x_2}{d})$, $n=1,2,3,\cdots$. Here, $x_1$-axial wave numbers $\beta^{\pm}_n$ satisfy that
\begin{equation}\begin{split}\label{wavenumbereq1}
-(1-M^2)\beta^2-2kM\beta+k^2=\frac{n^2 \pi^2}{d^2}.
\end{split}\end{equation}
Let $K_0=\frac{kd}{\pi\sqrt{1-M^2}}$. Let us denote that $[K_0]$ is the integer part of $K_0$. We set $[K_0]=N_0$. Then $x_1$-axial wave numbers $\beta^{\pm}_n$ are given by
\begin{equation}\begin{split}\label{wavenumber1}
\beta^{\pm}_n=&\frac{-kM\pm \sqrt{k^2-\frac{n^2 \pi^2}{d^2}(1-M^2)}}{1-M^2},\text{ if $n\leq N_0$},\\
\beta^{\pm}_n=&\frac{-kM\pm i\sqrt{\frac{n^2 \pi^2}{d^2}(1-M^2)-k^2}}{1-M^2},\text{ if $n\geq N_0$}.
\end{split}\end{equation}
The solutions $p^{+}_n(x)$ or $p^{-}_n(x)$ are outgoing waves in the direction of $x_1\rightarrow \infty$ ($x_1\rightarrow -\infty$, respectively). If $\beta^{\pm}_n$ is real, then $p^{\pm}_n$ are called propagating modes. The group velocity $\frac{\partial \omega}{\partial \beta}$ is positive for $p^{+}_n$ and negative for $p^{-}_n$.  The number of propagating modes increases as the Mach number $M\rightarrow 1$.  In the convected Helmholtz equation, it is well-known that inverse upstream modes $p^{+}_n(x)$ which have a positive group velocity and a negative phase velocity $\frac{\omega}{\beta}$ is appeared in \cite{BBL2004}. If $\beta^{\pm}_n$ is complex, then $p^{\pm}_n$ is called an evanescent mode. Evanescent modes $p^{+}_n$ or $p^{-}_n$ are exponentially decaying when $x_1\rightarrow \infty$ ($x_1\rightarrow -\infty$, respectively).

To construct the posed problem from (\ref{oed1}), some preliminaries are concerned in the following. Let us choose a bounded domain $\Omega_{b}=\{x\in\Omega,\; x^{-}< x_1 < x^{+}\}$ which contains the support $\Omega_f$ of $f(x)+\dot{W}^{h}(x)$. Define $\Sigma_{\pm}$ are two boundaries $\{x\in\Omega|\;x_1=x^{\pm}\}$. The inner product $(\cdot,\cdot)_{L^2(\Sigma_{+})}$ (respectively $(\cdot,\cdot)_{L^2(\Sigma_{-})}$) is defined by
\begin{equation}\begin{split}
(u,v)_{L^2(\Sigma_{+})}=\int_{\Sigma_{+}}u\bar{v} \;dx_2.
\end{split}\end{equation}

We now discuss the reduced problem in $\Omega_b$ which is equivalent to $(\ref{oed1})$. According to \cite{BBL2004}, the problem (\ref{oed1}) is reduced to
\begin{equation}\begin{split}\label{oedreduce1}
&(1-M^2)\frac{\partial^2 p^{h}_b}{\partial x^2_1}+\frac{\partial^2 p^{h}_b}{\partial x^2_2}+2ikM\frac{\partial p^{h}_b}{\partial x_1}+k^2p^{h}_b=f+\dot{W}^{h}\;\;\;\;\;\; \text{in }\Omega_{b},\\
&\frac{\partial p^{h}_b}{\partial x_2}=0\;\;\;\;\;\; \text{on }\partial\Omega \cap \partial\Omega_{b},\\
&\frac{\partial p^{h}_b}{\partial {\bf n}}=-T_{\pm}p^{h}_b\;\;\;\; \text{on }\Sigma_{\pm},
\end{split}\end{equation}
where the vector ${\bf n}$ denotes the outward unit normal to $\Sigma_{\pm}$. Here, the Dirichlet-to-Neumann(DtN) operators $T_{\pm}:H^{1/2}(\Sigma_{\pm})\rightarrow H^{-1/2}(\Sigma_{\pm})$ are defined by
\begin{equation}\begin{split}\label{dtn1}
T_{\pm}(\psi)=\mp\sum^{\infty}_{n=0}i\beta^{\pm}_n (\psi,\phi_n)_{L^2(\Sigma_{\pm})}\phi_n(x_2).
\end{split}\end{equation}
The DtN operators $T_{\pm}$ play as exact nonreflecting boundary condition so that two problems (\ref{oed1}) and (\ref{oedreduce1}) are equivalent. Furthermore, a weak from of (\ref{oedreduce1}) is written as follows:
\begin{equation}\begin{split}\label{woed11}
a_{\Omega_b}(p^{h}_b,q)=-\int_{\Omega_b}(f+\dot{W}^{h})\bar{q}\;dx,\;\;\;\;\; \forall q\in H^{1}(\Omega_b),
\end{split}\end{equation}
where the sesquilinear form $a_{\Omega_b}(\cdot,\cdot)$ is
\begin{equation}\begin{split}
a_{\Omega_b}(p,q)=b_{\Omega_b}(p,q)+C_{\Omega_b}(p,q)
\end{split}\end{equation}
with
\begin{equation}\begin{split}
b_{\Omega_b}(p,q)&=\int_{\Omega_b}\int_{\Omega}(\mathcal{M}\;\nabla p)\cdot\nabla\bar{q}\;dx+\int_{\Omega}k^2 p\bar{q}\;dx+\langle T_{+}p,q\rangle_{\Sigma_{+}}+\langle T_{-}p,q\rangle_{\Sigma_{-}},\\
c_{\Omega_b}(p,q)&=-2ik\int_{\Omega}\nabla\cdot \left(\begin{array}{c} Mp\\0 \\\end{array}\right)\bar{q}\;dx.
\end{split}\end{equation}
Here, the brackets $\langle \cdot,\cdot\rangle_{\Sigma_{+}}$, $\langle \cdot,\cdot\rangle_{\Sigma_{-}}$ are the natural duality pairing of $H^{1/2}(\Sigma_{\pm})$ and $H^{-1/2}(\Sigma_{\pm})$, respectively. The well-posedness of the problem (\ref{oedreduce1}) is served in the next Theorem.
\begin{thm}\label{t1} We assume that
\begin{equation}\begin{split}\label{assumptionk}
k\neq \sqrt{1-M^2}\;\frac{n \pi}{d}
\end{split}\end{equation}
for $n\in\mathds{N}$. Then the problem (\ref{woed1}) is well posed.
\end{thm}
\begin{proof}
Refer to Theorem 2.2 in \cite{BBL2004}.
\end{proof}

We now start with reducing the PML problem (\ref{oedreduce3}). After doing this, the error analysis can be derived from the reduced problem. More precisely, our aim is to compare the solution $p^{h}_b$ of (\ref{oedreduce1}) with the solution $p^{L}_b$ which satisfies:
\begin{equation}\begin{split}\label{oedreduce4}
&(1-M^2)(\frac{\partial }{\partial x_1}+\frac{Mki}{1-M^2})^2 p^{L}_b+\frac{\partial^2 p^{L}_b}{\partial x^2_2}+\frac{k^2}{1-M^2}p^{L}_b=f+\dot{W}^{h}\;\;\;\;\;\; \text{in }\Omega_{b},\\
&\frac{\partial p^{L}_b}{\partial x_2}=0\;\;\;\;\;\; \text{on }\partial\Omega \cap \partial\Omega_{b},\\
&\frac{\partial p^{L}}{\partial {\bf n}}=-T^{L}_{\pm}p^{L}_b\;\;\;\; \text{on }\Sigma_{\pm},
\end{split}\end{equation}
where the vector ${\bf n}$ denotes the outward unit normal to $\Sigma_{\pm}$ . Here, the DtN operators $T^{L}_{\pm}:H^{1/2}(\Sigma_{\pm})\rightarrow H^{-1/2}(\Sigma_{\pm})$ are defined by
\begin{equation}\begin{split}\label{dtn2}
T^{L}_{\pm}(\psi)&=\mp\sum^{\infty}_{n=0}i\nu^{\pm}_n (\psi(x_1,\cdot),\phi_n(\cdot))_{L^2(\Sigma_{\pm})}\phi_n(x_2),\\
\nu^{+}_n &= \beta^{+}_{n}- \frac{\beta^{+}_n-\beta^{-}_n}{1-e^{-i(\beta^{+}_n-\beta^{-}_n)\int^{L}_{0}\frac{1}{\alpha(x^{+}+s)}ds}},\\
\nu^{-}_n &= \beta^{-}_{n}+ \frac{\beta^{+}_n-\beta^{-}_n}{1-e^{-i(\beta^{+}_n-\beta^{-}_n)\int^{L}_{0}\frac{1}{\alpha(x^{-}-s)}ds}}.\\
\end{split}\end{equation}

Let us prove that the solutions of problems (\ref{oedreduce3}) and (\ref{oedreduce4}) are equivalent in $\Omega_b$, i.e. $p^{L}|_{\Omega_b}=p^{L}_b$. Then by using $p^{h}_b$ and $p^{L}_b$, we can obtain the error estimation with respect to $p^{h}$ and $p^{L}$ in $\Omega_b$. The next Theorem asserts that $p^{L}|_{\Omega_b}=p^{L}_b$.
\begin{thm}\label{thm2}
 Suppose $p^{L}$ and $p^{L}_b$ are satisfied (\ref{oedreduce3}) and (\ref{oedreduce4}), respectively. Then $p^{L}|_{\Omega_b}=p^{L}_b$.
\end{thm}
\begin{proof} The key idea is to find a correct boundary condition on $\Sigma_{\pm}$. In $\Omega^{L}_{\pm}$, the solution $p^{L}$ satisfies a homogeneous equation. Therefore, we can apply the separation of variable method. Let us consider the solution in the domain $\Omega^{L}_{+}$. We define
\begin{equation}\begin{split}\label{modal}
\psi^{\pm}_{n}(x_1)=e^{-\frac{Mki}{1-M^2}(x_1-x^{+})+i(\beta^{\pm}_n+\frac{Mk}{1-M^2})\int^{x_1}_{x^{+}} \frac{1}{\alpha(s)}ds}.
\end{split}\end{equation}
Then $\psi^{\pm}_{n}$ satisfies
\begin{equation}\begin{split}\label{psi1}
\alpha(x_1)(\frac{\partial}{\partial x_1}+\frac{Mki}{1-M^2})\psi^{\pm}_{n}(x_1)=i(\beta^{\pm}_n+\frac{Mk}{1-M^2})\psi^{\pm}_{n}(x_1).
\end{split}\end{equation}
It follows from (\ref{psi1}) that $\psi^{\pm}_{n}(x_1)\phi_n(x_2)$ satisfies the homogeneous equation problem of (\ref{oedreduce3}) with boundary condition on $\partial\Omega \cap \partial\Omega_{b}$. By the Dirichlet boundary condition on $\Sigma^{L}_{+}$, the solution $p^{L}_{+}$ in $\Omega^{L}_{+}$ is given by
\begin{equation}\begin{split}\label{pl1}
p^{L}_{+}(x)=\sum^{\infty}_{n=0}(p^{L}_{+}(x^{+},\cdot),\phi_n(\cdot))_{\Sigma_{+}}(A^{+}_n\psi^{+}_n(x_1)+A^{-}_n\psi^{-}_n(x_1))\phi_{n}(x_2),
\end{split}\end{equation}
where
\begin{equation}
A^{\pm}_n=\mp\frac{e^{i\beta^{\mp}_n\int^{L}_0\frac{1}{\alpha(x^{+}+s)}ds}}{e^{i\beta^{+}_n\int^{L}_0\frac{1}{\alpha(x^{+}+s)}ds}-e^{i\beta^{-}_n\int^{L}_0\frac{1}{\alpha(x^{+}+s)}ds}}.
\end{equation}
By (\ref{assumptionk}) and (\ref{alpha}), $(\beta^{+}_n-\beta^{-}_n)\int^{L}_0\frac{1}{\alpha(x^{+}+s)}ds$ has a positive imaginary part $\int^{L}_{0}\frac{\sigma(x^{+}+s)}{w}ds$. Therefore, it can not be a element of $2\pi \mathds{Z}:=\{2\pi z|\;z\in\mathds{Z}\}$. In other words, the denominator of $A^{\pm}_n$ is not zero. In this reason, $A^{\pm}_n$ and the formula (\ref{pl1}) are well-defined.

Since $\alpha(x_1)|_{\Sigma_{+}}=1$ and $\alpha(x_1)$ is continuous, an exact boundary condition on $\Sigma_{+}$ satisfies that
\begin{equation}\begin{split}\label{dtnre1}
\frac{\partial p^{L}_{b}}{\partial x_1}|_{\Sigma_{+}}&=\alpha(x_1)(\frac{\partial }{\partial x_1}+\frac{Mki}{1-M^2})p^{L}_{b}|_{\Sigma_{+}}-\frac{Mki}{1-M^2}p^{L}_{b}|_{\Sigma_{+}}\\
&=\sum^{\infty}_{n=0}(p^{L}_{+}(x^{+},\cdot),\phi_n(\cdot))_{\Sigma_{+}}(iA^{+}_n(\beta^{+}_n+\frac{Mk}{1-M^2})+iA^{-}_n(\beta^{-}_n+\frac{Mk}{1-M^2}))\phi_{n}(x_2)\\
&\;\;\;-\sum^{\infty}_{n=0}(p^{L}_{+}(x^{+},\cdot),\phi_n(\cdot))_{\Sigma_{+}}(iA^{+}_n\frac{Mk}{1-M^2}+iA^{-}_n\frac{Mk}{1-M^2})\phi_{n}(x_2)\\
&=i\sum^{\infty}_{n=0}(p^{L}_{b}(x^{+},\cdot),\phi_n(\cdot))_{\Sigma_{+}}(A^{+}_n\beta^{+}_n+A^{-}_n\beta^{-}_n)\phi_{n}(x_2).
\end{split}\end{equation}
Repeating the argument for the boundary condition on $\Sigma_{+}$, an exact boundary condition on $\Sigma_{-}$ is
\begin{equation}\begin{split}\label{dtnre2}
\frac{\partial p^{L}_{b}}{\partial x_1}|_{\Sigma_{-}}=-i\sum^{\infty}_{n=0}(p^{L}_{b}(x_-,\cdot),\phi_n(\cdot))_{\Sigma_{+}}(B^{+}_n\beta^{+}_n+B^{-}_n\beta^{-}_n)\phi_{n}(x_2),
\end{split}\end{equation}
where
\begin{equation}
B^{\pm}_n=\mp\frac{e^{i\beta^{\mp}_n\int^{L}_0\frac{1}{\alpha(x^{-}-s)}ds}}{e^{i\beta^{+}_n\int^{L}_0\frac{1}{\alpha(x^{-}-s)}ds}-e^{i\beta^{-}_n\int^{L}_0\frac{1}{\alpha(x^{-}-s)}ds}}.
\end{equation}
Let $\nu^{+}_n=A^{+}_n\beta^{+}_n+A^{-}_n\beta^{-}_n$ and $\nu^{-}_n=(B^{-}_n\beta^{+}_n+B^{+}_n\beta^{-}_n)$. Then, we complete the proof.
\end{proof}

A weak from of (\ref{oedreduce4}) is also written by
\begin{equation}\begin{split}\label{woed3}
a^{L}_{\Omega_b}(p,q)=-\int_{\Omega_b}(f_0+\dot{W}^{h})\bar{q}\;dx,\;\;\;\;\; \forall q\in H^{1}(\Omega_b),
\end{split}\end{equation}
where the sesquilinear form $a^{L}_{\Omega}(\cdot,\cdot)$ is
\begin{equation}\begin{split}
a^{L}_{\Omega_b}(p,q)=b^{L}_{\Omega_b}(p,q)+C^{L}_{\Omega_b}(p,q)
\end{split}\end{equation}
with
\begin{equation}\begin{split}
b^{L}_{\Omega_b}(p,q)&=\int_{\Omega_L}(\mathcal{M}\;\nabla p)\cdot\nabla\bar{q}dx-\int_{\Omega_L}k^2 \;p\bar{q}\;dx+\langle T^{L}_{+}p,q\rangle_{\Sigma_{+}}+\langle T^{L}_{-}p,q\rangle_{\Sigma_{-}},\\
c^{L}_{\Omega_b}(p,q)&=-2ik\int_{\Omega_L}\nabla\cdot \left(\begin{array}{c} Mp\\0 \\\end{array}\right)\bar{q}\;dx.
\end{split}\end{equation}
{\bf Remark.} Suppose $\psi^{+}_n(x_1)$ are outgoing propagating modes, i.e. $n\leq N_0$. Then $\psi^{-}_n(x_1)$ are modes reflected by $\Sigma^{L}_{+}$. Amplitudes of the reflection coefficients $R^{\sigma}_n$ for each $\psi^{-}_n(x_1)$ are
\begin{equation}\label{rf1}
R^{\sigma}_n=\left|\frac{A^{-}_n}{A^{+}_n}\right|=e^{-\frac{2k}{1-M^2}\sqrt{1-\frac{n^2}{K^2_0}}\int^{L}_{0}\frac{\sigma(x^{+}+s)}{\omega}ds}.
\end{equation}
For evanescent modes $\psi^{+}_n(x_1)$ ($n\geq N_0$), amplitudes of reflection coefficients $R^{\sigma}_n$ are
\begin{equation}\label{rf2}
R^{\sigma}_n=\left|\frac{A^{-}_n}{A^{+}_n}\right|=e^{-\frac{2kL}{1-M^2}\sqrt{\frac{n^2}{K^2_0}-1}}.
\end{equation}
In other words, reflection coefficients of waves caused by $\Sigma^{L}_{+}$ are bounded by $e^{-\frac{2k}{1-M^2}C\int^{L}_{0}(1\wedge\frac{\sigma(x^{+}+s)}{\omega})ds}$, where $C$ is a positive constant depending on $K_0$. For example, if $\sigma$ is that of the example $\ref{exam1}$, then upper bounds of reflection coefficients are
\begin{equation}
e^{-\frac{2C}{3c_0(1-M^2)}\sigma_{+} L^3}\;\text{ and $e^{-\frac{2kC}{1-M^2} L}$}
\end{equation}
for propagating modes and evanescent modes, respectively. By using (\ref{rf1}) and (\ref{rf2}), we check that the solution $p^{L}$ converges to the solution $p^{h}$ as $L\rightarrow\infty$ in the following. Suppose the length $L$ of layers goes to infinity. Then amplitudes of reflection coefficients converge to 0 by (\ref{rf1}) and (\ref{rf2}). This means that all modes $\psi^{-}_n(x_1)$ reflected by PML are exponentially attenuated. The rigorous proof of the convergence will be studied in the next subsection.

Furthermore, the well-posedness of the problem (\ref{oedreduce4}) follows from Theorem \ref{thmwell} and (\ref{modal}). The  well-posedness of the problem is shown in the next theorem.
\begin{thm}\label{ttt1} Suppose that (\ref{assumptionk}) holds. Then the problem (\ref{oedreduce4}) is well posed.
\end{thm}
\begin{proof}
Applying Theorem \ref{thmwell}, it is enough to show that the homogeneous problem of (\ref{oedreduce4}) has only trivial solution. By applying (\ref{modal})  in Theorem \ref{thm2} and replacing $x^{+}$ with $0$, the solution of the homogeneous problem is decomposed as follows:
\begin{equation}\label{hode}
p^{L}_b(x)=\sum^{\infty}_{n=0}(D^{+}_n\psi^{+}_{n}(x_1)+D^{-}_n\psi^{-}_{n}(x_1))\phi_n(x_2),
\end{equation}
where $D^{\pm}_n$ are complex constants. The boundary condition $\frac{\partial p^{L}}{\partial {\bf n}}=-T^{L}_{\pm}p^{L}_b$ on $\Sigma^{L}_{\pm}$ implies that
\begin{equation}
D^{+}_n(\nu^{\pm}_n-\beta^{+}_n)=D^{-}_n(\nu^{\pm}_n-\beta^{-}_n)=0\;\;\;\;\text{for $n=0,1,2,\cdots$}.
\end{equation}
Since $(\nu^{\pm}_n-\beta^{+}_n)=(\nu^{\pm}_n-\beta^{-}_n)=0$ are equivalent to $(\beta^{+}_n-\beta^{-}_n)=0$, the assumption $(\ref{assumptionk})$ leads to $D^{\pm}_n=0$, for $n=0,1,2,\cdots$. Consequently, $p^{L}_b(x)=0$.
\end{proof}

\subsection{Convergence and error estimates in $\Omega_b$} Recall that the object of our study is to establish the error analysis with respect to $p$ of (\ref{oe1}) and $p^{L}$ of (\ref{oedreduce3}). As discussed in section 2, we applied the discretized white noise for building more regular problem (\ref{oed1}) and provided the error estimation $E[|p-p^{h}|^{2}_{L^2(\Omega_b)}]$. Furthermore, it was shown from (\ref{oedreduce1}) and Theorem \ref{thm2} that $p^{h}$ and $p^{L}$ are equivalent to $p^{h}_b$ and $p^{L}_b$ in $\Omega_b$, respectively. Therefore, we would be able to obtain the error between $p^{h}$ and $p^{L}$ which leads to the final result in Theorem \ref{thm4}. In this reason, we first study the error analysis with respect to $p^{h}_b$ and $p^{L}_b$. To do this, some preliminary observations are discussed in the followings.

Let us consider $p^{h}_b$ and $p^{L}_b$ which satisfy variational formulas:
\begin{equation}\begin{split}
a_{\Omega_b}(p^{h}_b,q)=&-\int_{\Omega_b}(f_0+\dot{W}^{h})\bar{q}\;dx,\;\;\;\;\; \forall q\in H^{1}(\Omega_b),\\
a^{L}_{\Omega_b}(p^{L}_b,q)=&-\int_{\Omega_b}(f_0+\dot{W}^{h})\bar{q}\;dx,\;\;\;\;\; \forall q\in H^{1}(\Omega_b),
\end{split}\end{equation}
respectively. We also define a function $\tilde{\sigma}(x_1)=(1\wedge\frac{\sigma(x_1)}{\omega})$. In this setting, the next lemma is proved.
\begin{lem}\label{lemerror2}
Suppose the length $L$ of layers is large enough. Then there exists positive constants $C_{1}$ and $C_2$ such that for all $p,q\in H^{1}(\Omega_b)$, we have
\begin{equation}\label{varierror}
|a_{\Omega_b}(p,q)-a^{L}_{\Omega_b}(p,q)|\leq C_1(e^{-C_2 \int^{L}_{0}\tilde{\sigma}(x^{+}+s)ds}+e^{-C_2 \int^{L}_{0}\tilde{\sigma}(x^{-}-s)ds})\|p\|_{H^{1}(\Omega_b)}\|q\|_{H^{1}(\Omega_b)}.
\end{equation}
Precisely, $C_2=\frac{2k}{1-M^2}(1 \wedge \sqrt{\frac{(N_0+1)^2}{K^2_0}-1})$.
\end{lem}
\begin{proof}
From (\ref{woed1}) and (\ref{woed2}), it follows that
\begin{equation}\label{113}
|a_{\Omega_b}(p,q)-a^{L}_{\Omega_b}(p,q)|\leq |\langle T_{+}p,q\rangle_{\Sigma_{+}}-\langle T^{L}_{+}p,q\rangle_{\Sigma_{+}}|+|\langle T_{-}p,q\rangle_{\Sigma_{-}}-\langle T^{L}_{-}p,q\rangle_{\Sigma_{-}}|.
\end{equation}
Since the estimation of two terms in the right-hand side of $(\ref{113})$ are analogous, we only derive the upper bound of the first term $|\langle T_{+}p,q\rangle_{\Sigma_{+}}-\langle T^{L}_{+}p,q\rangle_{\Sigma_{+}}|$. Let $\zeta=p|_{\Sigma_+}$ and $\eta=q|_{\Sigma_+}$. From (\ref{dtn1}) and (\ref{dtn2}), it is derived that for $\zeta\in H^{1/2}(\Sigma_+)$,
\begin{equation}
(T_+ - T^{L}_+)\zeta=-\sum^{\infty}_{n=0}i(\beta^{+}_{n}-\nu^{+}_{n})\zeta_n \phi_n(x_2),
\end{equation}
where $\zeta_n(x_1) =(\zeta(x_1,\cdot),\phi_n(\cdot))_{L^2(\Sigma_+)}$. Thus, it is shown that
\begin{equation}
\langle(T_+ - T^{L}_+)\zeta,\eta\rangle_{\Sigma_{+}}=-\sum^{\infty}_{n=0}i(\beta^{+}_{n}-\nu^{+}_{n})\zeta_n \bar{\eta}_n,
\end{equation}
where $\eta_n =(\eta,\phi_n)_{L^2(\Sigma_+)}$. Hence, (\ref{alpha}) and (\ref{dtn2}) lead to the inequality
\begin{equation}\begin{split}\label{eq114}
|\langle(T_+ - T^{L}_+)\zeta,\eta\rangle_{\Sigma_{+}}|\leq&\sum^{\infty}_{n=0}|\beta^{+}_{n}-\nu^{+}_{n}||\zeta_n \bar{\eta}_n|=\sum^{\infty}_{n=0}\frac{|\beta^{+}_{n}-\beta^{-}_{n}|}{|1-e^{i(\beta^{-}_{n}-\beta^{+}_{n})(L+i\int^{L}_{0}\frac{\sigma(x^{+}+s)}{\omega}ds)}|}|\zeta_n \bar{\eta}_n|.
\end{split}\end{equation}

On the other hand, if $z\in\mathds{C}$, $\mathrm{Im}(z)<0$, and $|\mathrm{Im}(z)|$ is large enough, then the following inequality holds:
\begin{equation}\label{exp}
|1-e^{iz}|\geq |e^{-\mathrm{Im}(z)}-1|\geq \frac{1}{2}e^{-\mathrm{Im}(z)}.
\end{equation}
To apply the inequality (\ref{exp}) for the term $|1-e^{i(\beta^{-}_{n}-\beta^{+}_{n})(L+i\int^{L}_{0}\frac{\sigma(x^{+}+s)}{\omega}ds)}|$ in (\ref{eq114}), we consider two cases in the following:

${}$\\
\rom{1}. propagating modes : suppose $\beta^{+}_{n}-\beta^{-}_{n}=\frac{2k}{1-M^2}\sqrt{1-\frac{n^2}{K^2_0}}$ i.e. $n\leq N_0$. Then, $\mathrm{Im}((\beta^{-}_{n}-\beta^{+}_{n})(L+i\int^{L}_{0}\frac{\sigma(x^{+}+s)}{\omega}ds))=-\frac{2k}{1-M^2}\sqrt{1-\frac{n^2}{K^2_0}}\int^{L}_{0}\frac{\sigma(x^{+}+s)}{w}ds$. Clearly, this quantity is negative. By choosing large enough L, the inequalily (\ref{exp}) holds for propagating modes. Therefore, it holds that
\begin{equation}\begin{split}\label{propa}
|\beta^{+}_{n}-\nu^{+}_{n}|&\leq \frac{4k}{1-M^2}\sqrt{1-\frac{n^2}{K^2_0}}\;e^{-\frac{2k}{1-M^2}\sqrt{1-\frac{n^2}{K^2_0}}\int^{L}_{0}\frac{\sigma(x^{+}+s)}{\omega}ds)}\\
&\leq \frac{4k}{1-M^2}e^{-\frac{2k}{1-M^2}\int^{L}_{0}\frac{\sigma(x^{+}+s)}{\omega}ds)}.
\end{split}\end{equation}

${}$\\
\rom{2}. evanescent modes : suppose $\beta^{+}_{n}-\beta^{-}_{n}=\frac{2ik}{1-M^2}\sqrt{\frac{n^2}{K^2_0}-1}$ i.e. $n\geq N_0+1$. Then, $\mathrm{Im}((\beta^{-}_{n}-\beta^{+}_{n})(L+i\int^{L}_{0}\frac{\sigma(s)}{\omega}ds))=-\frac{2kL}{1-M^2}\sqrt{\frac{n^2}{K^2_0}-1}$. By choosing large enough L, the inequality (\ref{exp}) is verified for evanescent modes. Thus,
\begin{equation}\begin{split}\label{evan}
|\beta^{+}_{n}-\nu^{+}_{n}|&\leq \frac{4k}{1-M^2}\sqrt{\frac{n^2}{K^2_0}-1}\;e^{-\frac{2kL}{1-M^2}\sqrt{\frac{n^2}{K^2_0}-1}}\\
&\leq \frac{4k}{1-M^2}\frac{n}{K_0}e^{-\frac{2kL}{1-M^2}\sqrt{\frac{(N_0+1)^2}{K^2_0}-1}}.
\end{split}\end{equation}
From (\ref{propa}) and (\ref{evan}), we obtain
\begin{equation}\begin{split}\label{eq115}
|\langle(T_+ - T^{L}_+)\zeta,\eta\rangle_{\Sigma_{+}}|\leq&\frac{4k}{1-M^2}\sum^{N_0}_{n=0}e^{-\frac{2k}{1-M^2}\int^{L}_{0}\frac{\sigma(x^{+}+s)}{\omega}ds)}|\zeta_n \bar{\eta}_n|\\
&+\frac{4k}{1-M^2}\sum^{\infty}_{n=N_0+1}\frac{n}{K_0}e^{-\frac{2kL}{1-M^2}\sqrt{\frac{(N_0+1)^2}{K^2_0}-1}}|\zeta_n \bar{\eta}_n|\\
\leq&\frac{4k}{1-M^2}e^{-\frac{2k}{1-M^2}(1 \wedge \sqrt{\frac{(N_0+1)^2}{K^2_0}-1})\int^{L}_{0}\tilde{\sigma}(x^{+}+s)ds}\sum^{\infty}_{n=0}(1+\frac{n^2}{K^2_0})^{1/2}|\zeta_n \bar{\eta}_n|\\
\leq& C_1 e^{-C_2\int^{L}_{0}\tilde{\sigma}(x^{+}+s)ds}|\zeta|_{H^{1/2}(\Sigma_{+})}|\bar{\eta}|_{H^{1/2}(\Sigma_{+})}.
\end{split}\end{equation}
According to the trace Theorem in \cite{E1998}, it follows that
\begin{equation}\begin{split}\label{tt12}
|\langle(T_+ - T^{L}_+)&p,q\rangle_{\Sigma_{+}}|\leq C_1e^{-C_2 \int^{L}_{0}\tilde{\sigma}(x^{+}+s)ds}\|p\|_{H^{1}(\Omega_b)}\|q\|_{H^{1}(\Omega_b)}.
\end{split}\end{equation}
By repeating the procedure for $(\ref{tt12})$, a similar result for $|\langle(T_- - T^{L}_-)p,q\rangle_{\Sigma_{-}}|$ is shown as follows:
\begin{equation}\begin{split}
|\langle(T_- - T^{L}_-)&p,q\rangle_{\Sigma_{-}}|\leq C_1e^{-C_2 \int^{L}_{0}\tilde{\sigma}(x^{-}-s)ds}\|p\|_{H^{1}(\Omega_b)}\|q\|_{H^{1}(\Omega_b)}.
\end{split}\end{equation}
\end{proof}
Now, we turn to the error estimation $\|p_b^h - p_b^L\|_{H^1(\Omega_b)}$. Let $V = H^1(\Omega_b)$. Then linear operators $\mathcal{A}$ and $\mathcal{A}^{L}$ in $\mathcal{L}(V,V^{\prime})$ are defined by
\begin{equation}\begin{split}
\langle \mathcal{A}p,q\rangle_{V^{\prime}}&=a_{\Omega_b}(p,q),\\
\langle \mathcal{A}^{L}p,q\rangle_{V^{\prime}}&=a^{L}_{\Omega_b}(p,q),
\end{split}\end{equation}
for all $p,q\in V$. Here, the operator $\mathcal{A}$ is the same as that of p.424 in \cite{BBL2004}. By Riesz representation Theorem, $\mathcal{A}$ and $\mathcal{A}^{L}$ are uniquely determined. Moreover, operator norms $\|\mathcal{A}\|_{\mathcal{L}(V,V^{\prime})}$ and $\|\mathcal{A}^{L}\|_{\mathcal{L}(V,V^{\prime})}$ are bounded by upper bounds of forms $a_{\Omega_b}(\cdot,\cdot)$ and $a^{L}_{\Omega_b}(\cdot,\cdot)$, respectively. It follows from Lemma \ref{lemerror2} that
\begin{equation}\begin{split}\label{opererror}
\|\mathcal{A}-\mathcal{A}^{L}\|_{\mathcal{L}(V,V^{\prime})}\leq C_1(e^{-C_2 \int^{L}_{0}\tilde{\sigma}(x^{+}+s)ds}+e^{-C_2 \int^{L}_{0}\tilde{\sigma}(X^{-}-s)ds}).
\end{split}\end{equation}
Since $p^{h}_b$ and $p^{L}_b$ satisfies that for all $q\in V$,
\begin{equation}\begin{split}
\langle \mathcal{A}p^{h}_b,q\rangle_{V^{\prime}}&=-\langle\tilde{f},q\rangle_{L^{2}(\Omega)},\\
\langle \mathcal{A}^{L}p^{L}_b,q\rangle_{V^{\prime}}&=-\langle\tilde{f},q\rangle_{L^{2}(\Omega)},
\end{split}\end{equation}
respectively, it holds that
\begin{equation}\begin{split}\label{oper1}
\langle\mathcal{A}^{L}(p^{h}_b-p^{L}_b),q\rangle_{V^{\prime}}=\langle(\mathcal{A}^{L}-\mathcal{A})p^{h}_b,q\rangle_V^{\prime}
\end{split}\end{equation}
for all $q\in V$.

From (\ref{opererror}) and (\ref{oper1}), the error estimation for $\|p^{h}_b-p^{L}_b\|_{H^1(\Omega_b)}$ is established as follows.
\begin{thm}\label{thm3}
There exists a large positive constant $\tilde{L}$ such that for all $L\geq\tilde{L}$, $\mathcal{A}^{L}$ is an isomorphism on $H^1(\Omega_b)$. Moreover, we have
\begin{equation}\label{finalerror2}
\|p^{h}_b-p^{L}_b\|_{H^1(\Omega_b)}\leq C_1 (e^{-C_2 \int^{L}_{0}\tilde{\sigma}(x^{+}+s)ds}+e^{-C_2 \int^{L}_{0}\tilde{\sigma}(x^{-}-s)ds})\|p^{h}_b\|_{H^{1}(\Omega_b)},
\end{equation}
where $C_{1}$ and $C_2$ are positive constants depending on $k$ and $M$.
\end{thm}
\begin{proof}
The argument of the proof is similar to that of Theorem 4.4 in \cite{BBL2004}. The operator $\mathcal{A}^{L}$ can be rewritten as $\mathcal{A}+(\mathcal{A}^{L}-\mathcal{A})$. According to \cite{BBL2004}, $\mathcal{A}$ is an isomorphism on the Hilbert space $V$ . Therefore,
\begin{equation}\label{aa}
\langle Ap,q\rangle_{V^{\prime}}=\langle p,A^{-1}q\rangle_{V}.
\end{equation}
We consider the problem: finding $u\in V$ which satisfies
\begin{equation}\label{ppp0}
\langle\mathcal{A}^{L}u,q\rangle_{V^{\prime}}=\langle g,q \rangle_{V^{\prime}},
\end{equation}
for $g\in V^{\prime}$ and all $q\in V$. From (\ref{aa}), the problem (\ref{ppp0}) becomes
\begin{equation}\label{fp}
\langle (I+\mathcal{A}^{-1}(\mathcal{A}^{L}-\mathcal{A}))u,q\rangle_{V}=\langle \mathcal{A}^{-1}g,q \rangle_{V}.
\end{equation}
From(\ref{opererror}), it holds that for $L\geq\tilde{L}$,
\begin{equation}
\|\mathcal{A}^{L}-\mathcal{A}\|<\|\mathcal{A}^{-1}\|^{-1}.
\end{equation}
This leads to
\begin{equation}
\|\mathcal{A}^{-1}(\mathcal{A}^{L}-\mathcal{A})\|<1,
\end{equation}
where $\|\cdot\|$ is a formal operator norm. By the Banach fixed point theorem in \cite{E1998}, the lineear map $\mathcal{A}^{L}$ admits a unique solution. Furthermore, it follows that
\begin{equation}\label{oper13}
\|(I+\mathcal{A}^{-1}(\mathcal{A}^{L}-\mathcal{A}))^{-1}\|<\frac{1}{1-\|\mathcal{A}^{-1}(\mathcal{A}^{L}-\mathcal{A})\|}.
\end{equation}
By (\ref{fp}) and (\ref{oper13}), the following inequality holds:
\begin{equation}\label{oper14}
\|u\|_V<\frac{\|\mathcal{A}^{-1}g\|_V}{1-\|\mathcal{A}^{-1}(\mathcal{A}^{L}-\mathcal{A})\|}.
\end{equation}
Let us set $u=p^{h}_b-p^{L}_b$. Then, by (\ref{oper1}) and (\ref{oper14}), we obtain
\begin{equation}\begin{split}
\|p^{h}_b-p^{L}_b\|_{H^{1}(\Omega_b)}&<\frac{\|\mathcal{A}^{-1}(\mathcal{A}^{L}-\mathcal{A})p^{h}_b\|_V}{1-\|\mathcal{A}^{-1}(\mathcal{A}^{L}-\mathcal{A})\|}\\
&\leq\frac{\|\mathcal{A}^{-1}\|\|\mathcal{A}^{L}-\mathcal{A}\|\|p^{h}_b\|_V}{1-\|\mathcal{A}^{-1}(\mathcal{A}^{L}-\mathcal{A})\|}.
\end{split}\end{equation}
Since $\|\mathcal{A}^{L}-\mathcal{A}\|$ is small enough by choosing some $\tilde{L}$, it is possible to taking a positive constant $C>1$ as a upper bound of $\frac{1}{1-\|\mathcal{A}^{-1}(\mathcal{A}^{L}-\mathcal{A})\|}$. Consequently, the proof is completed from (\ref{opererror}).
\end{proof}

Theorem \ref{thm3} implies that $p^{L}$ converge to $p^{h}$ in $\Omega_b$ as $L\rightarrow \infty$. Applying two estimations (\ref{finalerror1}) and (\ref{finalerror2}), it follows that the next error estimation holds.

\begin{thm}\label{thm4}
Suppose that (\ref{assumptionk}) holds. Let the length $L$ of the layers is large enough. Then there exist positive constants $C_{1}$ and $C_2$ depending on $k$, $\sup|K_i|$ and $M$ such that for $0<\epsilon<1$, we obtain
\begin{equation}\label{final}
E[\|p-p^{L}\|^2_{L^2(\Omega_b)}]\leq C_1(h^{2-\epsilon}+e^{-C_2 \int^{L}_{0}\tilde{\sigma}(x^{+}+s)ds}+e^{-C_2 \int^{L}_{0}\tilde{\sigma}(x^{-}-s)ds}).
\end{equation}
\end{thm}
\begin{proof}
By the triangle inequality, it follows that
\begin{equation}\begin{split}
E[\|p-p^{L}\|^2_{L^2(\Omega_b)}]=&E[\|(p-p^{h})-(p^{h}-p^{L})\|^2_{L^2(\Omega_b)}]\\
\leq&2 E[\|p-p^{h}\|^2_{L^2(\Omega_b)}]+2E[\|p^{h}_b-p^{L}_b\|^2_{L^2(\Omega_b)}]\\
\leq&2 E[\|p-p^{h}\|^2_{L^2(\Omega_b)}]+2E[\|p^{h}_b-p^{L}_b\|^2_{H^1(\Omega_b)}].
\end{split}\end{equation}
Therefore, by Theorem \ref{thm1} and Theorem \ref{thm3}, the next inequality is proved:
\begin{equation}\begin{split}
E[\|p-p^{L}\|^2_{L^2(\Omega_b)}]\leq C_1(h^{2-\epsilon}+e^{-C_2 \int^{L}_{0}\tilde{\sigma}(s)ds}).
\end{split}\end{equation}
\end{proof}

Theorem \ref{thm4} asserts that the solution $p^{L}$ of the problem (\ref{woed2}) converges to $p$ in $\Omega_b$ as $L\rightarrow \infty$ and $h\rightarrow 0$. Therefore, we obtain an approximated solution of $p$ by solving the variational problem (\ref{woed2}).

\section*{Conclusion}
We have studied the stochastic convected Helmholz equation in an infinite duct. Since the regularity of the solution of SPDEs is generally weak, an alternative problem has been constructed by using the discretized white noise. In this setting, we have proposed the modified PML model which omits the presence of inverse upstream modes. Applying modal expansion approach, an error analysis of the PML model has been provided. Finally, the stochastic PML-truncation error of the solution has been established in the sense of  $E[\|\cdot\|^2_{L^2}]$.

\section*{Appendix}
${}$\\
{\bf The proof of Lemma \ref{lem2}}
\begin{proof}
By (\ref{gf}) and (\ref{g1}), we obtain
\begin{equation}\begin{split}
|G(x,y)-G(x,z)|\leq& 2|\Phi(x,y)-\Phi(x,z)|+\sum^{\infty}_{n=1}|\Phi(x,e^{+}_{n}+y)-\Phi(x,e^{+}_{n}+z)|\\
&+\sum^{\infty}_{n=1}|\Phi(x,e^{-}_{n}+y)-\Phi(x,e^{-}_{n}+z)|\\
&+\sum^{\infty}_{n=1}|\Phi(x,e^{+}_{n}-y)-\Phi(x,e^{+}_{n}-z)|\\
&+\sum^{\infty}_{n=1}|\Phi(x,e^{-}_{n}-y)-\Phi(x,e^{-}_{n}-z)|\\
=:& \, 2I_0+I_{1}+I_{2}+I_{3}+I_{4}.
\end{split}\end{equation}
To prove the inequality (\ref{ien1230}), we will derive the following inequalities:
\begin{equation}\begin{split}
\int_{\Omega_b}I^2_i\;dx&\leq C |y-z|^{2-\epsilon},\;\;\;\;\;i=0,1,2,3,
\end{split}\end{equation}
where $C$ is a positive constant depending on $M$ and $k$.

First, the case of $I_0$ is considered. From Lemma \ref{lem1} and the boundedness of the domain $\Omega_b$, it follows that
\begin{equation}\begin{split}
\int_{\Omega_b}I^2_0 \;dx&\leq C(\int_{\Omega_b}|\ln{\rho(x-y)}-\ln{\rho(x-z)}|^2 \;dx+\int_{\Omega_b}|\tilde{V}(x,y)-\tilde{V}(x,z)|^2 \;dx)\\
&\leq C\int_{\Omega_b}|\ln{\rho(x-y)}-\ln{\rho(x-z)}|^2 \;dx+C|y-z|^2.
\end{split}\end{equation}
For $0<\epsilon<1$, the integral $\int_{\Omega_b}|\ln{\rho(x-y)}-\ln{\rho(x-z)}|^2 \;dx$ is reformulated by
\begin{equation}\begin{split}\label{eq001}
\int_{\Omega_b}|\ln{\rho(x-y)}-\ln{\rho(x-z)}|^{\epsilon}(\rho(x-y)-\rho(x-z))^{2-\epsilon}\left(\int^{1}_{0}\frac{1}{\theta \rho(x-y)+(1-\theta)\rho(x-z)}d\theta\right)^{2-\epsilon} \;dx.
\end{split}\end{equation}
By the triangular inequality and H\"older inequality,  (\ref{eq001}) is bounded by
\begin{equation}\begin{split}\label{sep1}
C&|y-z|^{2-\epsilon}\int_{\Omega_b}|\ln{\rho(x-y)}-\ln{\rho(x-z)}|^{\epsilon}\left(\int^{1}_{0}\frac{1}{\theta \rho(x-y)+(1-\theta)\rho(x-z)}d\theta\right)^{2-\epsilon} \;dx\\
&\leq C|y-z|^{2-\epsilon}\int_{\Omega_b}|\ln{\rho(x-y)}-\ln{\rho(x-z)}|^{\epsilon}\left(\frac{1}{\rho(x-y)}+\frac{1}{\rho(x-z)}\right)^{2-\epsilon} \;dx\\
&\leq C|y-z|^{2-\epsilon}\int_{\Omega_b}|\ln{\rho(x-y)}-\ln{\rho(x-z)}|^{\epsilon}\left(\frac{1}{|x-y|}+\frac{1}{|x-z|}\right)^{2-\epsilon} \;dx\\
&\leq C|y-z|^{2-\epsilon}\left(\int_{\Omega_b}|\ln{\rho(x-y)}-\ln{\rho(x-z)}|^{3}dx\right)^{\epsilon/3}\left(\int_{\Omega_b}|\frac{1}{|x-y|}+\frac{1}{|x-z|}|^{\frac{3(2-\epsilon)}{3-\epsilon}}dx\right)^{(3-\epsilon)/3}.
\end{split}\end{equation}
Since $\rho(x-y)$ and $\rho(x-z)$ are bounded by some $\tilde{C}>0$ for $x,y,z\in \Omega_b$, the integral $\int_{\Omega_b}|\ln{\rho(x-y)}-\ln{\rho(x-z)}|^{3}dx$ has an upper bound
\begin{equation}\begin{split}\label{sep2}
C\int_{\Omega_b}|\ln{\rho(x-y)}|^{3}dx\leq C\int^{\tilde{C}}_{0}r\ln^3 r \;dr< \infty.
\end{split}\end{equation}
Moreover, the integral $\int_{\Omega_b}|\frac{1}{|x-y|}+\frac{1}{|x-z|}|^{\frac{3(2-\epsilon)}{3-\epsilon}}dx$ is also bounded by
\begin{equation}\begin{split}\label{sep3}
C\int_{\Omega_b}|\frac{1}{|x-y|}|^{\frac{3(2-\epsilon)}{3-\epsilon}}dx\leq C\int^{\tilde{C}}_{0}r^{-\frac{3-2\epsilon}{3-\epsilon}} dr < \infty.
\end{split}\end{equation}
By (\ref{sep1}), (\ref{sep2}), and (\ref{sep3}), we conclude that
\begin{equation}\begin{split}\label{i01}
\int_{\Omega_b}I^2_0 \;dx&\leq C |y-z|^{2-\epsilon},
\end{split}\end{equation}
where $C$ is a positive constant depending on $M$ and $k$.

On the other hand, Hankel function $H^{1}_0(z)$ has an asymptotic behavior
\begin{equation}\begin{split}\label{hanasymp}
H^{1}_{0}(z)=\sqrt{\frac{2}{\pi z}}e^{i(z-\pi/4)}(1+\mathcal{O}(\frac{1}{z}))
\end{split}\end{equation}
for large real $z$. To prove $\int_{\Omega_b}I^2_1\;dx\leq C |y-z|^{2-\epsilon}$, we will use (\ref{hanasymp}) in the following. We choose a large $n_0>0$ such that $\rho(x-y-e^{+}_{n})$ and $\rho(x-z-e^{+}_{n})$ have a lower bound $\tilde{C}>0$.  Then, $\Phi(x,e^{+}_{n}+y)$ and $\Phi(x,e^{+}_{n}+z)$ have an asymptotic behavior (\ref{hanasymp}) for all $n\geq n_0$. Choosing sufficiently large $n_0$, it follows that
\begin{equation}\begin{split}\label{i12}
|\Phi&(x,e^{+}_{n}+y)-\Phi(x,e^{+}_{n}+z)|\leq|H^{1}_{0}(k\rho(x-y-e^{+}_{n}))-H^{1}_{0}(k\rho(x-z-e^{+}_{n}))|\\
&\leq C |\frac{1}{\sqrt{\rho(x-y-e^{+}_{n})}}-\frac{1}{\sqrt{\rho(x-z-e^{+}_{n})}}|\\
&\leq C \frac{|\rho(x-y-e^{+}_{n})-\rho(x-z-e^{+}_{n})|}{\sqrt{\rho(x-y-e^{+}_{n})\rho(x-z-e^{+}_{n})}(\sqrt{\rho(x-y-e^{+}_{n})}+\sqrt{\rho(x-y-e^{+}_{n})})}\\
&\leq C \frac{|\rho(y-z)|}{n^3 d^3}\\
&\leq C \frac{|y-z|}{n^3 d^3},\\
\end{split}\end{equation}
where $C$ is a positive constant depending on $M$ and $k$. The inequality (\ref{i12}) implies that
\begin{equation}\begin{split}\label{i13}
\sum^{\infty}_{n=n_0}|\Phi(x,e^{+}_{n}+y)-\Phi(x,e^{+}_{n}+z)|&\leq \sum^{\infty}_{n=n_0}\frac{C}{n^3 d^3}|y-z|\leq C |y-z|.
\end{split}\end{equation}
For $n< n_0$, the same procedure as in the case of $I_0$ is available. Hence, we obtain
\begin{equation}\begin{split}\label{i11}
\int_{\Omega_b}|\Phi(x,e^{+}_{n}+y)-\Phi(x,e^{+}_{n}+z)|^2 \;dx&\leq C |y-z|^{2-\epsilon}, \;\;\;\;\text{for all $n< n_0$}.
\end{split}\end{equation}
By (\ref{i13}) and (\ref{i11}), for small $h>0$, it holds that
\begin{equation}\begin{split}\label{i1}
\int_{\Omega_b}I^2_1 \;dx&\leq C_1\sum^{n_0-1}_{n=1}\int_{\Omega_b}|\Phi(x,e^{+}_{n}+y)-\Phi(x,e^{+}_{n}+z)|^2 dx\\
&\;\;\;\;+C_2\int_{\Omega_b}|\sum^{\infty}_{n=n_0}(\Phi(x,e^{+}_{n}+y)-\Phi(x,e^{+}_{n}+z))|^2 dx\\
&\leq C_1 |y-z|^{2-\epsilon}+ C_2\int_{\Omega_b}|y-z|^2 dx\\
&\leq C |y-z|^{2-\epsilon},
\end{split}\end{equation}
where $y,z\in K_i$ and $C$ is a positive constant depending on $k$, $d$, $|\Omega_b|$, and $M$. $I_{2,3,4}$ cases are proved by the same method as in the case of $I_1$.

\end{proof}



\bibliographystyle{elsarticle-num}


\end{document}